\documentclass[11pt]{article} 

\usepackage[utf8]{inputenc} 

\usepackage[margin=1in]{geometry}%
\usepackage{graphicx}%
\usepackage{multirow}%
\usepackage{amsmath,amssymb,amsfonts}%
\usepackage{amsthm}%
\usepackage{mathrsfs}%
\usepackage[title]{appendix}%
\usepackage{xcolor}%
\usepackage{textcomp}%
\usepackage{manyfoot}%
\usepackage{booktabs}%
\usepackage{algorithm}%
\usepackage{algorithmicx}%
\usepackage{algpseudocode}%
\usepackage{listings}%
\usepackage{hyperref}
\usepackage[capitalize]{cleveref}

\def\R{\mathbb{R}}

\theoremstyle{thmstyleone}%
\newtheorem{theorem}{Theorem}
\newtheorem{lemma}[theorem]{Lemma}%

\theoremstyle{thmstyletwo}%
\newtheorem{remark}{Remark}%

\theoremstyle{thmstylethree}%

\title{Polynomial Inequalities and Optimal Stability of Numerical Integrators}

\author{Luke Shaw\footnote{Departament de Matemàtiques, Universitat Jaume I and ironArray SLU, {C/ Tirant Lo Blanc 6}, {Castell\'o de la Plana}, {12100}, {Castell\'o}, {Spain}} \\ Email: luke.shaw@ironarray.io}
\begin{document}
\maketitle

\abstract{A numerical integrator for $\dot{x}=f(x)$ is called \emph{stable} if, when applied to the 1D Dahlquist test equation $\dot{x}=\lambda x,\lambda\in\mathbb{C}$ with fixed timestep $h>0$, the numerical solution remains bounded as the number of steps tends to infinity. It is well known that no explicit integrator may remain stable beyond certain limits, depending on the domain of $\lambda$. Furthermore, these stability limits are only tight for certain specific integrators (different for each domain), which may then be called `optimally stable'. 
Such optimal stability results are typically proven using sophisticated techniques from complex analysis, leading to rather abstruse proofs. In this article, we pursue an alternative approach, exploiting connections with the Bernstein and Markov brothers inequalities for polynomials. This simplifies the proofs greatly and, moreover, offers a simple framework which unifies the diverse results that have been obtained.}

\section{Introduction}
For a very wide class of explicit numerical integrators for ordinary differential equations (ODEs) of the form $\dot{x}=f(x)$,  optimal stability limits for the Dahlquist test equation $\dot{x}=\lambda x$ have been proven in \cite{Jeltsch1981}. These limits, for an integrator using $m$ evaluations of $f(x)$ per step, correspond to the three cases of interest: 1) $|1+h\lambda/m|\leq1$ for arbitrary $\lambda\in\mathbb{C},\mathrm{Re}\{\lambda\}\leq0$;  2) $-2m^2\leq h\lambda\leq 0$ for $\lambda\in\mathbb{R}_-$; and 3) $-im\leq h\lambda\leq im$ for $\lambda\in i\mathbb{R}$. Proving such general results requires the use of complex analysis, which unfortunately encumbers the exposition.

However, if one considers a slightly less general family of integrators (still including all explicit Runge-Kutta methods), the stability analysis may be conducted via the analysis of polynomials. In this context, it has long been common to find the corresponding `optimal polynomials' simply stated as such, without proof of their optimality \cite{Guillou1960,Kinnmark1984,Sonneveld1985}. Such terseness also proves frustrating for researchers seeking a more pedagogical approach to this important special case.

The present article also restricts attention to integrators which may be analysed via considering polynomials, and we may thus avoid using complex analysis. We then detail extremely simple, \emph{expository} proofs which elucidate connections between the `optimal polynomials' and the Bernstein \cite{Malik1985} and Markov Brothers' inequalities \cite[pp. 89--91]{CheneyBook}. We are thus able to collect and unify results from multiple sources for the three cases of interest, depending on $\lambda$,  within the same framework, offering a resource of pedagogical value which is lacking in the current literature.

Our proofs, while less general, still apply to a great many integrators of practical interest, including all (explicit) Runge-Kutta methods. This is demonstrated by the fact that all theorems in this work offer simplified proofs of published results: \cref{thm:Euler} proves the result of \cite{Jeltsch1978} for explicit Runge-Kutta methods; \cref{thm:Verlet} (appropriately adapted) proves the result of \cite{Chawla1981}; and \cref{thm:Hyperbolic} unifies and clarifies the partial results of \cite{Kinnmark1984,Sonneveld1985,Vichnevetsky1983} and \cite[Sec. 4.3]{VanDerHouwenBook}.

\paragraph{Numerical Integration}
For the problem $f:D\subset\R^d\to\R^d$,
\begin{equation}
\frac{dx}{dt}\equiv\dot{x}=f(x),\quad x(0)=x_0, \quad t\in[0,T], \label{eq:ODE}
\end{equation}
with exact solution flow $\phi_t:\R^d\to\R^d$ one may design numerical integrators which are discrete maps that furnish approximations $\{x_n\}_{n=1}^N$ to the exact solution $x(t)$, $x_{n}\approx x(nh)$ at grid points $t_n=nh,n=1,\ldots N$ such that $t_N=T=Nh$. One such class of integrators are so-called (explicit) `one-step maps' (including all explicit Runge-Kutta methods) of the form
\begin{equation}\label{eq:OneStep}
x_{n+1}=\psi_h(x_n),\quad x_0=x(0),
\end{equation}
which generate the approximations at the grid points in an iterative fashion. Such an integrator is called \emph{consistent} if in the limit $h\to0$ $\| \psi_h(x)-\phi_h(x)\|=\mathcal{O}(h^{p+1}),\forall x\in\R^d$ with $p\geq 1$, which is a minimal condition typically required of any numerical integrator to be considered a `good' approximator. Consistency however, gives little clue as to the behaviour of the integrator for moderate timesteps, which may be more relevant in certain applications. The most important aspect of this behaviour is \emph{stability}.  Numerical instability refers to related undesirable phenomena which manifests when the integrator is operated with stepsizes $h$ which are too large: magnification of the rounding errors that occur in finite arithmetic; inordinate divergence from the exact solution; or failure to match the qualitative behaviour of the continuous dynamics \cref{eq:ODE} near fixed points. To attempt to understand these related issues, one considers the behaviour of the integrator via the aforementioned Dahlquist test problem.

\paragraph{Justification of the Dahlquist test problem}
We consider the relevance of the test problem to the one of the aforementioned phenomena: the magnification of rounding errors.  Denote the numerical solution implemented in exact arithmetic, $x_n$ and the computer implementation (with roundoff errors) $\widetilde{x}_n$. Then one has that the error $\widetilde{E}_n\equiv x_n- \widetilde{x}_n$ due to roundoff evolved under the exact flow follows
\begin{equation}\label{eq:RoundoffError}
\frac{d}{dt}\widetilde{E}_{n}=f(x_n)-f(\widetilde{x}_n)\approx \mathrm{\textbf{D}}(x_n-\widetilde{x}_n)= \mathrm{\textbf{D}}\widetilde{E}_{n},\quad \mathrm{\textbf{D}}\equiv\left.\frac{df}{dx}\right|_{x_n},
\end{equation}
where the approximation is `good' for $x_n$ close to $\widetilde{x}_n$ (i.e. small roundoff error, as one would expect). Consequently one has stability with respect to perturbations due to rounding errors for the exact flow if the eigenvalues of $\mathrm{\textbf{D}}$ all have non-positive real part. It is then desirable if the evolution under the numerical map $\psi_h$ applied to the linearised flow \cref{eq:RoundoffError} respects this, i.e. $\|\psi_h(\widetilde{E}_n)\|\leq\|\widetilde{E}_n\|$. If $\mathrm{\textbf{D}}$ is diagonalisable, \cref{eq:RoundoffError} reduces to $d$ copies of $\dot{E}=\lambda E, \lambda\in\mathbb{C}$ which is none other than the now standard test for (A-)stability (due to Dahlquist)
\begin{equation}\label{eq:Dahlquist}
\dot{x}=\lambda x,\quad x(0)=1,\qquad \{\lambda\in\mathbb{C}:\mathrm{Re}\{\lambda\}\leq0\}.
\end{equation}
Since diagonalisation commutes with almost all methods of practical interest, one may thus study stability via \cref{eq:Dahlquist} without loss of generality.

\paragraph{Setup} We consider consistent, explicit numerical integrators $\psi_h:\R^d\to\R^d$  applied with constant timestep $h$, which, when applied to \cref{eq:Dahlquist} give an iterative solution of the form
\begin{equation}\label{eq:StabPoly}
x_{n+1}=P(h\lambda)x_n,
\end{equation}
where $P(z)$ is a polynomial of degree $\leq m$ in $z$ with $P(z)=1+z+\mathcal{O}(z^2)$. This is the case for e.g. any consistent $m$-stage explicit Runge-Kutta (RK) method. In general a method using $m$ evaluations of $f$ (such as an $m$-stage RK method) gives a polynomial of degree $m$.  
Since a method using $m$ evaluations of $f$ has the same work as a method using 1 evaluation applied $m$ times with stepsize $h/m$, in the following we always compare methods of equivalent cost $m$ (i.e. polynomials of degree $m$) \cite{Jeltsch1981}.

It is then apparent from \cref{eq:StabPoly} that one avoids exponential instability (i.e exponential growth of the error $x_n$) if $\lvert P(h\lambda)\rvert\leq1$. One then calls the \emph{stability domain} of the method with \emph{stability polynomial} $P(z)$ the region
\begin{equation}\label{eq:StabilityDomain}
\mathcal{S}=\{z\in\mathbb{C}: |P(z)|\leq 1\}.
\end{equation}
For example, the Explicit Euler method for \cref{eq:ODE}, $x_{n+1}=x_n+hf(x_n)$, applied to the test problem \cref{eq:Dahlquist} gives $x_{n+1}=(1+h\lambda)x_n,$
and so the stability polynomial of (Explicit) Euler is $P_{Euler}(z)=1+z$, with stability domain the complex disc
$
\mathcal{D}_1=\{z\in\mathbb{C}: |1+z|\leq1\}.
$

Note that, on the other hand, the Implicit Euler method $x_{n+1}=x_{n}+hf(x_{n+1})$, applied to the test problem \cref{eq:Dahlquist} gives $x_{n+1}=(1-h\lambda)^{-1}x_n,$ which is stable $\forall\lambda\in\mathbb{C},\mathrm{Re}\{\lambda\}\leq0$. It is in fact the case that unlimited stability of this sort can only be possessed by implicit methods. However, implicit methods are more costly than explicit methods, and may be unsuitable for various reasons (e.g. high dimension), so that the study of the (limited) stability of explicit methods is eminently worthwhile. Given that their stability is limited, it is of interest to ask what are the ``optimal stability limits'' of explicit integrators which give rise to \cref{eq:StabPoly}, and which numerical integrators attain these limits. These limits also depend on the problem at hand, which induces certain restrictions on $\lambda$. {Namely, for integrators which give rise to stability polynomials of degree $m$ in $(\lambda h)$:}
\begin{itemize}
\item for general problems, $\lambda\in\mathbb{C},\mathrm{Re}\{\lambda\}\leq0$, the optimal stability region is $\mathcal{D}_m=\{\lambda h\in\mathbb{C}: |1+\lambda h/m|\leq 1\}$(\cref{thm:Euler});
\item for `parabolic problems' $\lambda\in\mathbb{R},\lambda\leq0$,  the optimal stability region is $-2m^2\leq\lambda h\leq0$ (\cref{thm:Verlet});
\item for `hyperbolic problems' $\lambda\in i\mathbb{R}$, the optimal stability region is $-i(m-1)\leq\lambda h\leq i(m-1)$ (\cref{thm:Hyperbolic}).
\end{itemize}
The names parabolic and hyperbolic come from the partial differential equations (PDE) context, in which upon discretising in space a parabolic (hyperbolic) PDE, one arrives at an equation of the form \cref{eq:ODE} with Jacobian having negative (pure imaginary) spectrum \cite[Sec. 1.4]{VanDerHouwenBook}. Note also that in the general and parabolic cases, the optimal limits are not enhanced by enlarging the family of integrators beyond those which give rise to \cref{eq:StabPoly}; only in the case of hyperbolic problems does one see improvement, to $-im\leq\lambda h\leq im$ (for a multistep method) \cite{Jeltsch1981}.
\section{General Theorem}
Two of the most prominent families of general numerical integrators  for \cref{eq:ODE} are Runge-Kutta methods (which are `one-step' in the sense of \cref{eq:OneStep}) and (linear) multi-step methods. We explicitly do not consider multi-step methods, since they do not give rise to maps of the form \cref{eq:StabPoly} when applied to \cref{eq:Dahlquist}. However, the result below in \cref{thm:Euler}, that applies to all Runge-Kutta methods is already quite strong, and was first shown in \cite{Jeltsch1978}, although in a less direct way.

\begin{theorem}[Optimal Stability - General]\label{thm:Euler}
Consider a consistent, explicit integrator for \cref{eq:ODE} which, when applied to \cref{eq:Dahlquist} generates a map of the form of \cref{eq:StabPoly}. Then its stability domain $\mathcal{S}_m$ contains the complex disc
$$
\mathcal{D}_m=\{z\in\mathbb{C}: |1+z/m|\leq1\},
$$
iff $P(z)=(1+z/m)^m$. This may be obtained by simply applying $m$ Explicit Euler steps in succession with stepsize $h/m$, i.e. $P_{mEuler}(z)=(1+z/m)^m$.
\end{theorem}
\begin{proof}
Since the stability polynomial $P(z)=1+z+\mathcal{O}(z^2)$ of degree $\leq m$ in $z$, one has
\begin{equation}\label{eq:ConsistencyRestricionRK}
P(0)=P'(0)=1.
\end{equation}
Consider now the bijective transformation $\mu=1+z/m$, which maps $\mathcal{D}_m$ to $\mathcal{D}_1$, so that $P(z)=\widetilde{P}(\mu)$. Since $\widetilde{P}$ is a polynomial, it is analytic (i.e. complex differentiable everywhere), as is the transformation $\mu$, and so we can differentiate and apply the chain rule safely, giving
$$
\frac{d}{dz}P(z)=\frac{d}{dz}\widetilde{P}(\mu)=\widetilde{P}'(\mu)\frac{d\mu}{dz}=\widetilde{P}'(\mu)/m.
$$
Since $P'(z=0)=1$, this implies that $\widetilde{P}'(\mu=1)=m$, which holds for any stability polynomial. We may then apply the Bernstein inequality \cite{Malik1985} (suggested on \cite[p.76]{Jeltsch1981})
\begin{equation}\label{eq:Bernstein}\tag{Bernstein}
 \max_{\vert \mu\vert=1}\vert \widetilde{P}'(\mu)\vert\leq m \max_{\vert \mu\vert=1}\vert \widetilde{P}(\mu)\vert,
\end{equation}
and insert the requirement that $\vert \widetilde{P}(\mu)\vert\leq1$ for stability, and since we know that for $\mu=1$ the derivative attains the value $m$, $\max_{\lvert \mu\rvert =1}\vert \widetilde{P}'(\mu)\vert\geq m$, giving
$$
m\leq\max_{\vert \mu\vert=1}\vert \widetilde{P}'(\mu)\vert\leq m \max_{\vert \mu\vert=1}\vert \widetilde{P}(\mu)\vert\leq m,
$$
and so one has equality $\max_{\vert \mu\vert=1}\vert \widetilde{P}'(\mu)\vert= m \max_{\vert \mu\vert=1}\vert \widetilde{P}(\mu)\vert$, which holds only in the case of the Bernstein polynomial $\widetilde{P}(\mu)=\mu^m$, i.e. $P(z)=(1+z/m)^m$ \cite[Sec. 2]{Malik1985}.
\end{proof}

\begin{remark}
The optimal stability of the Euler method, shown by  \cref{thm:Euler}, does in fact hold even when one allows competitor multi-step methods and other numerical integrators that do not give rise to \cref{eq:StabPoly}, although the proof requires intricate complex analysis \cite[Sec. 3]{Jeltsch1981}.
\end{remark}

\section{Parabolic Theorem}\label{sec:parabolic}
For `parabolic problems', where $\textbf{D}$ in \cref{eq:RoundoffError} has real spectrum, one is interested in \cref{eq:Dahlquist} with $\lambda\leq0,\lambda\in\R$. In this case, the domain of stability of relevance is the line segment $-a\leq\lambda h\leq0$. We show now that the maximal $a=2m^2$ for integrators, such as Runge-Kutta methods, giving rise to \cref{eq:StabPoly}, via essentially the same simple proof as that of \cref{thm:Euler}. 
{
\paragraph{Comments on the original proof} The first result regarding stability in this case is due to \cite{Guillou1960}. The problem amounts to finding the degreee $m$ polynomial of the form $P(z)=1+z+\ldots$ which has absolute value $\leq 1$ for $-a<z<0$ with $a>0$ maximal. It is advisable to start with the Chebyshev polynomials of the first kind $T_m(x)$ which are known to have maximal first derivative among all polynomials of degree $m$ which remain between $1$ and $-1$ on the interval $-1<x<1$. 
Starting from any other polynomial in this class, applying the bijective transformation $\mu = \mu(z)$ which maps $-1<\mu<1\rightarrow -a<z<0$ and ensuring that $P(0)=P'(0)=1$ hence gives a value for $a$ smaller than that obtained by starting with the Chebyshev polynomial. The advantage of the following proof is that it 1) draws a clear connection to the general case of \cref{thm:Euler} and 2) takes a less constructive approach, only introducing the Chebyshev polynomial at the end, as the only polynomial which saturates the Markov Brothers' Inequality.
}

\begin{theorem}[Optimal Stability - Parabolic]\label{thm:Verlet}
Consider a consistent, explicit integrator for \cref{eq:ODE} which, when applied to \cref{eq:Dahlquist} with $\lambda \leq0$ generates a map of the form of \cref{eq:StabPoly}.  Then it attains the optimal stability interval $-2m^2\leq z<0$ iff $P(z)=T_m(1+z/m^2)$, where $T_m(x)$ is the $m^{\text{th}}$ Chebyshev polynomial of the first kind.

\noindent The optimal polynomial may be obtained via a sequence of $m$ Euler steps of different stepsizes such that $P(z)=\prod_{i=1}^m(1+z/\xi_i)$, where $\xi_i=m^2\left[1-\cos\left(\frac{\pi(2i-1)}{2m}\right)\right]$ \cite{Guillou1960}.
\end{theorem}
\begin{proof}
Let $z=\lambda h$ as before, and consider now the transformation $\mu=1+z/m^2$, which maps $[-2m^2,0]$ to $[-1,1]$, so that $P(z)=\widetilde{P}(\mu)$. Differentiating $\widetilde{P}$ as before and using $P'(z=0)=1$ from consistency implies that $\widetilde{P}'(\mu=1)=m^2$. We may then apply the Markov Brothers inequality \cite[p.91]{CheneyBook} (suggested on \cite[p.79]{Jeltsch1981})
\begin{equation}\label{eq:Markov}\tag{Markov Brothers}
 \max_{-1\leq\mu\leq1}\vert \widetilde{P}'(\mu)\vert\leq m^2 \max_{-1\leq\mu\leq1}\vert \widetilde{P}(\mu)\vert,
\end{equation}
and insert the requirement that $\vert \widetilde{P}(\mu)\vert\leq1$ for stability, and since we know that for $\mu=1$ the derivative attains the value $m^2$, $\max_{-1\leq \mu\leq 1}\vert \widetilde{P}'(\mu)\vert\geq m^2$. Continuing as before, one has equality $\max_{-1\leq \mu\leq1}\vert \widetilde{P}'(\mu)\vert= m^2 \max_{-1\leq \mu\leq1}\vert \widetilde{P}(\mu)\vert$, which holds only in the case of the Chebyshev polynomial $\widetilde{P}(\mu)=T_m(\mu)$, i.e. $P(z)=T_m(1+z/m^2)$ \cite{Shadrin2004}.
\end{proof}

\begin{remark}
The optimal stability $-2m^2\leq z\leq0$ for $T_m(1+z/m^2)$ holds even when allowing for a broader class of competitor integrators, including linear multi-step methods, not only those giving rise to \cref{eq:StabPoly}. As for \cref{thm:Euler}, this extension of the result of \cref{thm:Verlet} requires the use of elaborate complex analysis \cite[Sec. 4]{Jeltsch1981}.
\end{remark}

\begin{remark} 
In fact \cref{thm:Verlet} also proves an optimality result for integrators for the second-order ODE $\ddot{x}=f(x)$, with test problem $\ddot{x}=-\omega^2 x, \omega\in\R$. Our result thus recovers the result of \cite{Chawla1981}, of optimal stability interval $h\omega\leq2m$, in a simpler way. In this context, the candidate methods are not Runge-Kutta but `splitting' methods, whose stability may be treated by a polynomial of degree $m$ in $z=(h\omega)^2$, $P(z)=1-z/2+\mathcal{O}(z^2)$ \cite{Blanes2008}. From \cref{thm:Verlet}, with transformation $\mu=1-z/(2m^2)$ the optimal polynomial is then
$P(z)=T_m(1-z/(2m^2))$ \cite[Secs. 4.4-4.5]{BouRabee2018}. Note that for the special case $\ddot{x}=-Ax+g(x)$, to which one can apply more exotic (splitting) integrators, the Chebyshev polynomials are also optimal, in an appropriate sense \cite{Casas2023}.
\end{remark}

\section{Hyperbolic Theorem}\label{sec:hyperbolic}
For `hyperbolic problems', where $\textbf{D}$ in \cref{eq:RoundoffError} has pure imaginary spectrum, one is interested in \cref{eq:Dahlquist} with $\lambda\in i\R$. In this case, the domain of stability of relevance is the line segment $-ia\leq\lambda h\leq ia$. We show now that the optimal $a=(m-1)$ for integrators, such as Runge-Kutta methods, giving rise to \cref{eq:StabPoly}.

{
\paragraph{Comments on the original proof} The first result regarding stability in this case is due to \cite{VanDerHouwenBook}. The proof proceeds by establishing an upper bound for the stability interval of $a = (m-1)$, via reformulation of the problem as a ``minimax problem for a class of rational
functions'', proposing a polynomial form, and showing that this form attains the maximum; unfortunately the method is only applicable for $m$ odd. \cite{Vichnevetsky1983} proved that for $m$ even the maximal possible stability interval is also $a = (m-1)$. \cite{Kinnmark1984, Sonneveld1985} then presented equivalent formulae for general $m$ which attain the optimal limits. The proof below brings out the similarities (and differences) between the hyperbolic case and the previous parabolic and general cases and proves both the optimal limits and the polynomials which attain them simultaneously, offering a succint summary of these disparate results with a method which serves for both $m$ odd and $m$ even. We remark that the relevant references for \cref{sec:parabolic,sec:hyperbolic} are collected in \cite[p.31-37]{HNWBook2}, with the optimal polynomials for the hyperbolic case given as the solution to Exercise 3 on p.37. Note that the result is stated for polynomials of degree $m+1$, and so is equivalent to the result of \cref{thm:Hyperbolic}.
}

Before proceeding, we first require a simple lemma.
\begin{lemma}\label{lem:Q}
Consider a polynomial $P(z)=1+z+\alpha z^2+\mathcal{O}(z^3)$ of degree $m$ in $z$, for some unknown $\alpha\in\mathbb{R}$,
such that $|P(i\sqrt{y})|\leq 1, y\in[0,a], a>0$. Then $\alpha\geq 1/2$.
\begin{proof}
Consider $P(i\sqrt{y}),y\in\mathbb{R}$ so that
$$
Q(y)=P(i\sqrt{y})P(-i\sqrt{y})=\left|P(i\sqrt{y})\right|^2=1+(1-2\alpha)y+\mathcal{O}(y^2),
$$
is a real polynomial of degree $m$ in $y$. Necessarily $\alpha\geq 1/2$ as otherwise $Q(y)>1$ (and consequently $|P(i\sqrt{y})|>1$) for all positive $y$ less than some value sufficiently close to 0. 
\end{proof}
\end{lemma}

\begin{theorem}[Optimal Stability - Hyperbolic]\label{thm:Hyperbolic}
Consider a consistent, explicit integrator for \cref{eq:ODE} which, when applied to \cref{eq:Dahlquist} with $\lambda\in i\mathbb{R}$ generates a map of the form of \cref{eq:StabPoly}.  Then it attains the optimal stability interval $-i(m-1)\leq z\leq i(m-1)$ iff 
\begin{equation}\label{eq:mgeneral}
P(z)=i^{m-1}T_{m-1}\left(\frac{z}{i(m-1)}\right)+i^{m-2}\left(1+\frac{z^2}{(m-1)^2}\right)U_{m-2}\left(\frac{z}{i(m-1)}\right),
\end{equation}
 {where $T_k (U_k)$ are the Chebyshev polynomials of degree $k$ of the first (second) kind.}
\end{theorem}
\begin{proof}
Let $z=\lambda h$ as before, and consider $P(z)=1+z+\alpha z^2+\mathcal{O}(z^3)$, for some unknown $\alpha\in\mathbb{R}$. One then has $P(i\sqrt{y}),y\in\mathbb{R}$ such that
$$
P(i\sqrt{y})=R+i\sqrt{y}I, \quad R(y)=1-\alpha y+\mathcal{O}(y^2), \quad I(y)=1+\mathcal{O}(y),
$$
where for $m=2k+1$ ($m=2k$), $k\in\mathbb{N}_0$, $R$ is of degree $k$ ($k$) in $y$ and $I$ is of degree 
$k$ ($k-1$). 
Defining $Q(y)=|P(i\sqrt{y})|^2=R^2+yI^2$ as in \cref{lem:Q}, it is clear that the maximal interval $0\leq y\leq a^2$ over which $Q(y)\leq 1$ is equal to or contained by those of $R(y)^2$ and $yI(y)$ separately.
We consider the cases of even $m=2k$ and odd $m=2k+1$ separately.

\paragraph{Case $m=2k+1$}
In this case we first optimise the interval $R(y)$ and then seek $\sqrt{y}I(y)$ such that the interval is preserved.
Using that $R(0)=1,R'(0)=-\alpha$ and the transformation $\mu=1-2y/a^2$ (which maps $0\leq y\leq a^2$ to $-1\leq \mu\leq 1$) to give $\widetilde{R}(\mu)=R(y)$, from \cref{thm:Verlet} one sees that the optimal polynomial is of the Chebyshev form and fulfills $\frac{\alpha a^2}{2}= k^2,$
and thus, since from \cref{lem:Q} one has $\alpha\geq 1/2$, one has the maximal $a^2=4k^2$ obtained for $R(y)=T_k\left(1-y/(2k^2)\right)$.

If we wish that $Q(y)\leq 1$, we require then that $\sqrt{y}I(y)=0$ when $R(y)^2=1$ (i.e. at the extrema $\xi_j=\cos(j\pi/k),j=0,1,\ldots,k$ of $T_k$), which is obtained by a polynomial of the form
$$
\sqrt{y}I(y)\propto \sqrt{y}\left(1-\frac{y}{4k^2}\right)\frac{d}{dy}T_k\left(1-\frac{y}{2k^2}\right),
$$
where one may use the fact that $T'_k(x)=kU_{k-1}(x)$, where $U_k$ are the Chebyshev polynomials of the second kind, and the condition $I(0)=1$ to determine the constant of proportionality, to give the optimal polynomial
$$
\sqrt{y}I(y)=\frac{\sqrt{y}}{k}\left(1-\frac{y}{4k^2}\right)U_{k-1}\left(1-\frac{y}{2k^2}\right),
$$
where $I(y)$ is indeed of degree $k$ in $y$. Taking $\sqrt{y}/2k\equiv\cos(\theta)$, one may then verify that $Q(y)=\cos^2(2k\theta)+\sin^2(\theta)\sin^2(2k\theta)\leq1, 0\leq y\leq 4k^2$ as required, using the properties of the Chebyshev polynomials $T_k(\cos(\theta))=\cos(k\theta),U_{k-1}(\cos(\theta))\sin(\theta)=\cos(k\theta)$.

Consequently, one may obtain stability for $-(m-1)i\leq z\leq (m-1)i$ where $m=2k+1$ for the polynomial (replacing $z=i\sqrt{y}$)  \cite[p. 257]{VanDerHouwenBook}
$$
P(z)=T_k\left(1+\frac{z^2}{2k^2}\right)+\frac{z}{k}\left(1+\frac{z^2}{4k^2}\right)U_{k-1}\left(1+\frac{z^2}{2k^2}\right).
$$
Note that by the same substitution $\sqrt{y}/2k\equiv\cos(\theta)$, one may express this polynomial as \cite{Kinnmark1984}
\begin{equation}\label{eq:modd}
P(z)=(-1)^k\left[T_{2k}\left(\frac{z}{2ki}\right)-i\left(1+\frac{z^2}{4k^2}\right)U_{2k-1}\left(\frac{z}{2ki}\right)\right].
\end{equation}

\paragraph{Case $m=2k$}
When $m$ is even one cannot proceed as before to design $I(y)$ since $I(y)$ may only be of degree $k-1$. Hence, we instead begin by optimising $\sqrt{y}I$, which is of degree $2k-1$ in $\sqrt{y}$, and then seek $R(y)$ such that the interval is preserved. We first apply the transformation $\mu=\sqrt{y}/(2k-1)$ (which maps $-(2k-1)\leq \sqrt{y}\leq (2k-1)$ to $-1\leq\mu\leq1$) to give $\widetilde{S}(\mu)=\sqrt{y}I(y)$, where $\widetilde{S}$ is degree $2k-1$ in $\mu$. One has from the condition $\sqrt{y}I(y)=\sqrt{y}+\mathcal{O}(y^{3/2})$ that $\widetilde{S}(\mu=0)=0,\widetilde{S}'(\mu=0)=(2k-1)$. We compare this to the polynomial $T_{2k-1}(\mu)$, and see that one also has $T_{2k-1}(0)=0, T'_{2k-1}(0)=(-1)^{k-1}(2k-1)$ (since $2k-1$ is odd) and thus 
$(-1)^{k-1}T_{2k-1}(\mu)-\widetilde{S}(\mu)$ possesses a double root at $\mu=0$. Then, since $\mu=0$ does not correspond to an extremum of $T_{2k-1}$, one may apply \cite[Prop. 3.3]{Casas2023} to show that $|\widetilde{S}(\mu)|>1$ on the interval  $-1\leq\mu\leq1$ unless it is equal to  $(-1)^{k-1}T_{2k-1}(\mu)$ and thus the optimal polynomial $\sqrt{y}I(y)=(-1)^{k-1}T_{2k-1}(\sqrt{y}/(2k-1))$.

We may now proceed in a similar fashion to the $m$ odd case above, requiring now that $R(y)=0$ at the extrema of $T_{2k-1}$, which may be achieved via a polynomial of the form
$$
R(y)\propto\left(1-\frac{y}{(2k-1)^2}\right)\frac{d}{d\sqrt{y}}T_{2k-1}\left(\frac{\sqrt{y}}{2k-1}\right),
$$
where one may use the fact that $T'_{2k-1}(x)=(2k-1)U_{2k-2}(x)$, and the condition $R(0)=1$ to determine the constant of proportionality, to give the optimal polynomial
\begin{equation}\label{eq:mevenR}
R(y)=(-1)^{k-1}\left(1-\frac{y}{(2k-1)^2}\right)U_{2k-2}\left(\frac{\sqrt{y}}{2k-1}\right),
\end{equation}
where $R(y)$ is indeed of degree $k$ in $y$.  Taking $\sqrt{y}/(2k-1)\equiv\cos(\theta)$, one may then verify that $Q(y)=\cos^2((2k-1)\theta)+\sin^2(\theta)\sin^2((2k-1)\theta)\leq1, 0\leq y\leq (2k-1)^2$ as required.
Consequently, one may obtain stability for $-(m-1)i\leq z\leq (m-1)i$ where $m=2k$ for the polynomial (replacing $z=i\sqrt{y}$) \cite[Thm. 3]{Kinnmark1984}
\begin{equation}\label{eq:meven}
P(z)=(-1)^{k-1}\left[\left(1+\frac{z^2}{(2k-1)^2}\right)U_{2k-2}\left(\frac{z}{i(2k-1)}\right)+iT_{2k-1}\left(\frac{z}{i(2k-1)}\right)\right].
\end{equation}
It is then possible to assimilate \cref{eq:modd,eq:meven} to the general form \cref{eq:mgeneral}.
\end{proof}

\begin{remark}
In distinction to the case for \cref{thm:Euler,thm:Verlet}, integrators giving rise to stability polynomials are no longer optimal if one allows other competitor integrators; specifically, if one considers multi-step integrators, it is possible to reach $a=m$ \cite[Sec. 5]{Jeltsch1981}.
\end{remark}

\begin{remark}
Note also that for $m$ odd, one has $\alpha=1/2$ as mentioned above, so that the integrator is in fact second order (although this is not the case for $m$ even, as a Taylor expansion of \cref{eq:mevenR} shows).
\end{remark}

\subsection*{Funding}
LS has been supported by Ministerio de Ciencia e Innovación (Spain) through project PID2022-136585NB-C21, MCIN/AEI/10.13039/501100011033/FEDER, UE.
\subsection*{Acknowledgments} The author would like to thank Fernando Casas and Jesús María Sanz-Serna for helpful comments and discussions.

\end{document}